\documentclass[12pt]{article}
\usepackage{amsmath,amsfonts,amssymb,amsthm}
\newcommand{\R}{\mathbb{R}}
\newcommand{\N}{\mathbb{N}}
\newcommand{\Z}{\mathbb{Z}}

\newcommand{\ca}{\mathcal A}
\newcommand{\twosum}[2]{\sum_{\substack{#1\\#2}}}

\newtheorem{thm}{Theorem}[section]
\newtheorem{lem}[thm]{Lemma}
\begin{document}
\title{Almost prime values of binary forms with one prime variable}
\author{A.J. Irving\\
Mathematical Institute, Oxford}
\date{}
\maketitle

\begin{abstract}
By establishing an improved level of distribution we study almost primes of the form $f(p,n)$ where $f$ is an irreducible binary form over $\Z$.
\end{abstract}

\section{Introduction}

A well known problem in number theory is to show that if $f\in \Z[x]$ is an irreducible polynomial with $\deg f\geq 2$ then, provided the values of $f$ have no fixed prime divisor, there are infinitely many $n\in \Z$ for which $f(n)$ is prime.  This seems to be out of reach of current methods.  However, using sieves one can show that there are infinitely many $n\in \Z$ for which $f(n)$ has a small number of prime factors.  Let $P_r$ denote numbers with at most $r$ prime factors, counted with multiplicities, and let $k=\deg f$.  Richert \cite{richert} showed that there are infinitely many $n$ for which $f(n)$ is a $P_{k+1}$.  An even harder question is to ask whether there are infinitely many primes $p$ for which $f(p)$ is itself prime.  This was also considered by Richert who showed that there are infinitely many $p$ for which $f(p)$ is a $P_{2k+1}$, (provided we impose conditions on $f$ to avoid the obvious counterexamples).

Both problems are made easier if we consider irreducible binary forms $f\in \Z[x,y]$ instead of single variable polynomials.  A theorem of Fermat states that any prime $p\equiv 1\pmod 4$ is the sum of two squares and therefore the binary quadratic form $m^2+n^2$ represents infinitely many primes.  The case of a general binary quadratic form was handled by Dirichlet.  Much more recently, Heath-Brown \cite{rhbcube} showed that the  cubic $m^3+2n^3$ represents infinitely many primes.  If $f$ is a binary form with $k\geq 4$ then the best result known is due to Greaves \cite{greaves} who showed that if $f$ is irreducible then the values $f(m,n)$ are infinitely often $P_{[k/2]+1}$, provided of course that they have no fixed prime divisor.  In this paper we will consider the values $f(p,n)$ of a binary form where $n$ is an integer and $p$ a prime.  A result of Fouvry and Iwaniec \cite{fouvryiwaniec} shows that there are infinitely many primes of the form $p^2+n^2$; we are unaware of any existing results dealing with higher degree forms.  It is clear that by fixing the prime variable $p$ and applying the above result of Richert to the resulting polynomial values we can obtain infinitely many $P_{k+1}$.  We will improve this result for all $k\geq 3$ as follows.

\begin{thm}\label{almostprimeform}
Let $f\in \Z[x,y]$ be an irreducible binary form of degree $k\geq 3$.   Suppose that for every prime $p$ we have 
$$\#\{n\pmod p:f(1,n)\equiv 0\pmod p\}<p.$$ 
There are then infinitely many pairs $(p,n)$ with $n\in\Z$ and $p$ prime for which $f(p,n)$ is a $P_{[3k/4]+1}$.
\end{thm}

The proof of this depends on an improved ``level of distribution'' result for the values $f(p,n)$.  Roughly speaking, we count the number of these which are divisible by an integer $d$ when $p$ and $n$ have size $N$.  If we were to consider each prime $p$ separately then we could only handle $d\leq N^{1-\delta}$ for any $\delta>0$.  We will show that we can obtain a result on average over $d$ provided that $d\leq N^{4/3-\delta}$.  Theorem \ref{almostprimeform} then follows easily by using the weighted sieve.  The  details of our level of distribution are somewhat technical so we will leave a precise statement until Section \ref{sec:lod}.

Our level of distribution  should be compared with Fouvry and Iwaniec's for the values $p^2+n^2$ \cite[Lemma 4]{fouvryiwaniec}.  In our notation their result  essentially states that one can take $d$ as large as $N^{2-\delta}$ for that form.  Their proof depends crucially on the fact that the roots of the congruence $n^2+1\equiv 0\pmod d$ satisfy very strong distribution properties.  This enables them to prove a large sieve inequality for the fractions $n/d$ which is essentially optimal.  Our result also depends on a large sieve type inequality.  However we do not have comparable distribution estimates for the roots of higher degree polynomial congruences and therefore our level of distribution is weaker.  In the next section we will give details of the variant of the large sieve we use.  It concerns the sum of a sequence of coefficients $\alpha_m$, for example the indicator function of the primes, over the points $(m,n)$ in a sublattice of $\Z^2$.  We will show that if we average over a suitable family of lattices then we can control such a sum.  To reduce the binary form question to one concerning lattices we use methods similar to those of Daniel \cite{daniel}.  
 
Throughout this paper we use the notation $(a;b)$ for the highest common factor of the integers $a$ and $b$.  We write $x\sim y$ for the inequality $y\leq x<2x$.  The notation $\|x\|$ denotes the Euclidean length of a vector $x\in\R^2$.  We will denote the indicator function of the primes by $\chi(n)$.  We fix a smooth function $W$ which has compact support in $[0,1]$ and which takes nonnegative values. Finally we adopt the standard convention that $\epsilon$ denotes a small positive quantity whose value may differ at each occurrence.  All our implied constants may depend on $\epsilon$, $W$  and the binary form $f$.

\subsection*{Acknowledgements}

This work was completed as part of my DPhil, for which I was funded by EPSRC grant EP/P505666/1. I am very grateful to the
EPSRC for funding me and to my supervisor, Roger Heath-Brown, for all
his valuable help and advice. 

\section{A Large Sieve for Lattices}

\subsection{Introduction}

Let $\alpha_m$ be a sequence of complex numbers with $|\alpha_m|\leq 1$ and let $\lambda\subseteq\Z^2$ be a lattice.  For $N\geq 0$ we are interested in the quantity 
$$\psi(\lambda,N,\alpha)=\sum_{(m,n)\in \lambda\cap (0,N]\times \Z}\alpha_mW(\frac{n}{N}).$$
We expect that for a typical $\lambda$ we have 
$$\psi(\lambda,N,\alpha)\approx \frac{N\hat W(0)}{\det \lambda}\sum_{m\leq N}\alpha_m.$$
We will show that this holds if we average over a suitable set of lattices $\lambda$.  We will only consider the case that the set of $m$-coordinates of points in $\lambda$:
$$\{m:(m,n)\in \lambda\}$$ 
has greatest common factor $1$, since if this does not hold then only a homogeneous arithmetic progression of $m$ occur so the result cannot be true.

We will write $\det\lambda=d$ and restrict our consideration to lattices with $d\sim D$ for some parameter $D$.  For a given lattice $\lambda$ we let $B_1$ be a nonzero element of $\lambda$ of minimal length and $B_2$ be a vector of minimal length in the elements of $\lambda$ which are not multiples of $B_1$.  It is well known that $\{B_1,B_2\}$ is a basis for $\lambda$ and that 
$$\|B_1\|\|B_2\|\asymp \det \lambda.$$
Let $B$ be the matrix with rows $B_1,B_2$.  Since we are free to choose the signs of both $B_1$ and $B_2$ we may assume that $B_{11}\geq 0$ and $\det B=\det \lambda$.  We know that $\|B_1\|\ll (\det\lambda)^{1/2}$ and thus we have the same bound for $B_{11}$ and $B_{12}$.  We will consider an average over lattices where each possible value for $B_{11}$ occurs at most once but we make no assumption on the distribution of the remaining entries in $B$.  Our result is then as follows.  It should be noted that the shortest nonzero vector in $\lambda$ may not be unique.  In this case we are free to choose the vector in such a way that the conditions of the theorem are satisfied.

\begin{thm}\label{latticelargesieve}
Let $\alpha_m$ be a sequence of complex numbers with $|\alpha_m|\leq 1$ and let $D,M_1\geq 1$.  Let $\Lambda$ be a set of lattices in $\Z^2$ such that if $\lambda\in \Lambda$ then $\det\lambda \sim D$ and, letting $B$ be as above, we have $B_{11}\sim M_1$.  Assume that for each $\lambda\in\Lambda$ the $m$-coordinates of points are coprime, (as described above).  In addition, suppose that for each $m\sim M_1$ we have 
$$\#\{\lambda\in\Lambda: B_{11}(\lambda)=m\}\leq 1.$$
Suppose that  $\delta>0$.  

\begin{enumerate}
\item If $D\leq N^{1-\delta}$ then for any $A>0$ we have 
$$\sum_{\lambda\in\Lambda}\left|\psi(\lambda,N,\alpha)-\frac{N\hat W(0)}{\det \lambda}\sum_{m\leq N}\alpha_m\right|\ll_{\delta,A} N^{-A}.$$

\item If 
$$N^{1-\delta}\leq D< M_1N^{1-\delta}$$
then
$$\sum_{\lambda\in\Lambda}\left|\psi(\lambda,N,\alpha)-\frac{N\hat W(0)}{\det \lambda}\sum_{m\leq N}\alpha_m\right|\ll_{\epsilon,\delta} N^{1+2\delta+\epsilon}M_1^{-1/2}D^{1/2}$$
for any $\epsilon>0$.
\end{enumerate}
\end{thm}

It is useful to know when this result is nontrivial.  We note that, since $\#\Lambda\ll M_1$, we have  
$$\sum_{\lambda\in\Lambda}\frac{N\hat W(0)}{\det \lambda}\sum_{m\leq N}\alpha_m\ll\frac{N^2M_1}{D}$$  
and that 
$$N^{1+2\delta+\epsilon}M_1^{-1/2}D^{1/2}<\frac{N^2M_1}{D}$$ 
if and only if 
$$D<N^{2/3-4\delta/3-2\epsilon/3}M_1.$$ 
Our bound can therefore only be nontrivial if $D\leq N^{2/3-\eta}M_1$ for some $\eta>0$.  In particular, since $M_1\ll D^{1/2}$ the largest $D$ we can handle is $D\ll N^{4/3-\eta}$.  However, if $M_1$ is smaller then the range of $D$ must be decreased.

\subsection{Transforming the Sum}

We can write 
$$\lambda=\{(u,v)B:(u,v)\in\Z^2\}.$$
Our assumption that the $m$-coordinates of points in $\lambda$ have greatest common factor $1$ implies that we must have $(B_{11};B_{21})=1$.  In addition, since $B_{11}\sim M_1\geq 1$ we have $B_{11}>0$.

For a fixed $m\in (0,N]$ we consider the quantity 
$$S(m)=\twosum{n\in\Z}{(m,n)\in \lambda}W(\frac{n}{N})=\twosum{(u,v)\in\Z^2}{B_{11}u+B_{21}v=m}W\left(\frac{B_{12}u+B_{22}v}{N}\right).$$
The condition 
$$m=B_{11}u+B_{21}v$$
is equivalent to 
$$m\equiv B_{21}v\pmod {B_{11}}$$
in which case 
$$u=\frac{m-B_{21}v}{B_{11}}.$$
We therefore have 
\begin{eqnarray*}
S(m)&=&\sum_{v\equiv m\overline{B_{21}}\pmod{B_{11}}}W\left(\frac{B_{12}(m-B_{21}v)+B_{11}B_{22}v}{B_{11}N}\right)\\
&=&\sum_{v\equiv m\overline{B_{21}}\pmod{B_{11}}}W\left(\frac{B_{12}m+dv}{B_{11}N}\right)\\
&=&\sum_{u\in\Z}W\left(\frac{B_{12}m+d(m\overline{B_{21}}+uB_{11})}{B_{11}N}\right)\\
&=&\sum_{u\in\Z}W\left(\frac{m(B_{12}+d\overline{B_{21}})}{B_{11}N}+\frac{du}{N}\right).\\
\end{eqnarray*}
We may now apply the Poisson summation formula to deduce that 
$$S(m)=\frac{N}{d}\sum_{v\in\Z}\hat W\left(\frac{vN}{d}\right)e\left(\frac{mv(B_{12}+d\overline{B_{21}})}{dB_{11}}\right).$$
We therefore conclude that 
$$\psi(\lambda,N,\alpha)=\frac{N}{d}\sum_{v\in\Z}\hat W\left(\frac{vN}{d}\right)\sum_{m\leq N}\alpha_me\left(\frac{mv(B_{12}+d\overline{B_{21}})}{dB_{11}}\right).$$
The $v=0$ term in this is 
$$\frac{N\hat W(0)}{d}\sum_{m\leq N}\alpha_m$$
which is precisely the main term we require.

For any $A\in\N$ we may integrate by parts $A$ times to obtain the standard estimate 
$$\hat W(x)\ll_A \min(1,|x|^{-A}).$$
Recall that we have $d\sim D$.  We will truncate the sum over $v$ to $|v|\leq DN^{-1+\delta}$.  Specifically, for any $\delta>0$ and $A\in\N$ we have 
$$\frac{N}{d}\sum_{|v|> DN^{-1+\delta}}\hat W\left(\frac{vN}{d}\right)\sum_{m\leq N}\alpha_me\left(\frac{mv(B_{12}+d\overline{B_{21}})}{dB_{11}}\right)\ll_{\delta,A}N^{-A}.$$

Combining all of the above we see that 
$$\psi(\lambda,N,\alpha)=\frac{N\hat W(0)}{d}\sum_{m\leq N}\alpha_m+\psi_1(\lambda,N,\alpha,\delta)+O_{\delta,A}(N^{-A})$$
where 
$$\psi_1(\lambda,N,\alpha,\delta)=\frac{N}{d}\sum_{0<|v|\leq DN^{-1+\delta}}\hat W\left(\frac{vN}{d}\right)\sum_{m\leq N}\alpha_me\left(\frac{mv(B_{12}+d\overline{B_{21}})}{dB_{11}}\right).$$
It remains to bound $\psi_1$, at least on average over $\lambda$.  This is trivial if 
$DN^{-1+\delta}<1$
that is 
$D<N^{1-\delta}$
as then $\psi_1=0$.  This is thus enough to prove the first assertion in Theorem \ref{latticelargesieve}.  We may therefore assume that $D\geq N^{1-\delta}$.  

We have 
$$\psi_1(\lambda,N,\alpha,\delta)\ll \frac{N}{D}\sum_{0<|v|\leq DN^{-1+\delta}}\left|\sum_{m\leq N}\alpha_me\left(\frac{mv(B_{12}+d\overline{B_{21}})}{dB_{11}}\right)\right|.$$
We will remove the factor 
$e\left(\frac{mvB_{12}}{dB_{11}}\right)$
using partial summation. This results in 
$$\psi_1(\lambda,N,\alpha,\delta)\ll \frac{N}{D}\left(1+N^{\delta}\frac{|B_{12}|}{B_{11}}\right)\sum_{0<|v|\leq DN^{-1+\delta}}\max_{N'\leq N}\left|\sum_{m\leq N'}\alpha_me\left(\frac{mv\overline{B_{21}}}{B_{11}}\right)\right|.$$
Recalling that $B_{12}\ll D^{1/2}$, $B_{11}\sim M_1\ll D^{1/2}$ and using our assumption that each $B_{11}$ occurs at most once we thus see that 
\begin{eqnarray*}
\lefteqn{\sum_{\lambda\in \Lambda}\psi_1(\lambda,N,\alpha,\delta)}\\
&\ll& N^{1+\delta}D^{-1/2}M_1^{-1}\sum_{B_{11}\sim M_1}\max_{(B_{21};B_{11})=1}\sum_{0<|v|\leq DN^{-1+\delta}}\max_{N'\leq N}\left|\sum_{m\leq N'}\alpha_me\left(\frac{mv\overline{B_{21}}}{B_{11}}\right)\right|.\\
\end{eqnarray*}
By Cauchy's inequality we may bound this by 
$$N^{1/2+3\delta/2}M_1^{-1/2}\psi_2(\Lambda,N,\alpha,\delta)^{1/2}$$
where
$$\psi_2(\Lambda,N,\alpha,\delta)=\sum_{B_{11}\sim M_1}\max_{(b;B_{11})=1}\sum_{0<|v|\leq DN^{-1+\delta}}\max_{N'\leq N}\left|\sum_{m\leq N'}\alpha_me\left(\frac{mvb}{B_{11}}\right)\right|^2.$$

\subsection{Applying the Large Sieve}

Each $\frac{vb}{B_{11}}$ occurring in $\psi_2$ is congruent mod $\Z$ to a unique $\frac{a}{q}$ with $(a;q)=1$, $0\leq a< q$ and $q\ll M_1$.  We will group together terms  with the same $a/q$ and bound the resulting sums over dyadic intervals $q\sim Q$.  We must therefore give an upper bound for the number of times each $\frac{a}{q}$ occurs in our sum.

\begin{lem}
Assume that $D,M_1\geq 1$ and $\delta>0$ satisfy 
$$N^{1-\delta}\leq D< M_1N^{1-\delta}.$$
Suppose that for each integer $B_{11}\sim M_1$ we are given an integer $b$ with $(b;B_{11})=1$.  Then, if $(a;q)=1$ and  $0\leq a<q\ll M_1$, we have  
$$\#\{B_{11}\sim M_1,0<|v|\leq DN^{-1+\delta}:\frac{vb}{B_{11}}\equiv\frac{a}{q}\pmod \Z\}=\begin{cases}
0 & q<N^{1-\delta}M_1D^{-1}\\
O(M_1q^{-1}) & \text{otherwise.}\\
\end{cases}$$ 
\end{lem}

\begin{proof}
If 
$$\frac{vb}{B_{11}}\equiv \frac{a}{q}\pmod \Z$$
with $(a;q)=1$ then since $(b;B_{11})=1$ we must have 
$$q=\frac{B_{11}}{(B_{11};v)}\geq \frac{B_{11}}{|v|}\geq M_1N^{1-\delta}D^{-1}.$$
This proves that there are no solutions if $q<N^{1-\delta}M_1D^{-1}$ so the first part of the lemma follows.

For the remainder of the proof we suppose that $q\geq N^{1-\delta}M_1D^{-1}$. If $(a;q)=1$ and  
$$\frac{vb}{B_{11}}\equiv \frac{a}{q}\pmod \Z$$  
then $q|B_{11}$.  It follows that 
$$vb\equiv aB_{11}/q\pmod{B_{11}}.$$
We therefore see that for given $q$ and $B_{11}$ the number of possible $v$ is
$O(DN^{-1+\delta}M_1^{-1}+1)$.  Moreover, since $q\mid B_{11}$ there
are $O(M_1q^{-1})$ possible $B_{11}$.
By assumption we know that 
$$DN^{-1+\delta}M_1^{-1}<1$$
so we may conclude that the quantity of interest is $O(M_1q^{-1})$ as required.
\end{proof}

Using the last lemma we deduce that  the part of $\psi_2$ with $q\sim Q$, for $Q\geq N^{1-\delta}M_1D^{-1}$, is bounded by 
$$M_1Q^{-1}\sum_{q\sim Q}\sum_{(a;q)=1}\max_{N'\leq N}\left|\sum_{m\leq N'}\alpha_m e(\frac{am}{q})\right|^2.$$
Applying a maximal form of the large sieve, as given by Montgomery
\cite{montmaxlarge}, we can majorise this by 
$$M_1Q^{-1}N(N+Q^2)=M_1N(Q^{-1}N+Q).$$
Recall that 
$$N^{1-\delta}M_1D^{-1}\ll Q\ll M_1$$
so our bound is at most 
$$M_1N(N^{\delta}M_1^{-1}D+M_1).$$
We have $M_1\ll D^{1/2}$ so the first term is always larger and the bound is simply 
$N^{1+\delta}D$.  This holds for all the dyadic intervals $q\sim Q$ under consideration so we conclude that for any $\epsilon>0$ we have 
$$\psi_2(\Lambda,N,\alpha,\delta)\ll_\epsilon N^{1+\delta+\epsilon}D$$
and therefore that 
$$\sum_{\lambda\in \Lambda}\psi_1(\lambda,N,\alpha,\delta)\ll_\epsilon N^{1+2\delta+\epsilon}M_1^{-1/2}D^{1/2}.$$
This completes the proof of Theorem \ref{latticelargesieve}.

\section{Level of Distribution}\label{sec:lod}

Rather than only considering the values $f(p,n)$ we will consider values $\alpha_m f(m,n)$ for sequences of complex numbers $\alpha_m$ with $|\alpha_m|\leq 1$.  Letting $\alpha_m$ be the indicator function of the primes will then recover the case in which we are most interested.  Our approach is able to handle any sequence $\alpha_m$ but there are a number of unpleasant technicalities to deal with.  To avoid this we will only consider $\alpha_m$ supported on primes $m$.  We will study the quantity 
$$A_d(N,\alpha)=\twosum{(m,n)\in (0,N]\times\Z}{f(m,n)\equiv 0\pmod d}\alpha_m W(\frac{n}{N}).$$
We expect that for $\alpha_m$ supported on primes we have,  at least on average over a suitable range of $d$, 
$$A_d(N,\alpha)\approx M_d(N,\alpha)$$
where 
$$M_d(N,\alpha)=\frac{N\nu(d)\hat W(0)}{d}\sum_{m\leq N}\alpha_m$$
and $\nu(d)$ is the number of solutions, $n$, of the congruence 
$$f(1,n)\equiv 0\pmod d.$$
We therefore wish to estimate the sum 
$$\sum_{d\sim D}|A_d(N,\alpha)-M_d(N,\alpha)|.$$

\begin{thm}\label{formlod}
Let $\alpha_m$ be a sequence of complex numbers with $|\alpha_m|\leq 1$ supported on prime values of $m$.  Suppose $\delta_1>0$ and $1\leq D\leq N^{4/3-\delta_1}$.  There exists a $\delta_2>0$ depending only on $\delta_1$ such that 
$$\sum_{d\sim D}|A_d(N,\alpha)-M_d(N,\alpha)|\ll_{\delta_1} N^{2-\delta_2}.$$
\end{thm}

The advantage of working with $\alpha_m$ supported on primes is that the contribution to our sum from points $(m,n)$ with $(m;d)>1$ is small.

\begin{lem}
Under the hypotheses of Theorem \ref{formlod} we have, for any $\epsilon>0$, that 
$$\sum_{d\sim D}\twosum{(m,n)\in (0,N]\times\Z}{(m;d)>1,f(m,n)\equiv
  0\pmod d}|\alpha_m| W(\frac{n}{N})\ll_\epsilon N^{1+\epsilon}.$$
\end{lem}

\begin{proof}
Since $\alpha_m$ is supported on primes the condition $(m;d)>1$ implies $m|d$.  We therefore have
\begin{eqnarray*}
\lefteqn{\sum_{d\sim D}\twosum{(m,n)\in
    (0,N]\times\Z}{(m;d)>1,f(m,n)\equiv 0\pmod d}|\alpha_m|
  W(\frac{n}{N})}\\
&\ll&\sum_{d\sim D}\twosum{m\leq N}{m|d}\chi(m)\twosum{n\leq N}{f(m,n)\equiv 0\pmod d}1\\
&=&\sum_{m\leq N}\chi(m)\twosum{n\leq N}{f(m,n)\equiv 0\pmod m}\#\{d|f(m,n):d\sim D,m|d\}\\
&\leq &\sum_{m\leq N}\chi(m)\twosum{n\leq N}{f(m,n)\equiv 0\pmod m}\tau(f(m,n))\\
&\ll_\epsilon&\sum_{m\leq N}\chi(m)\twosum{n\leq N}{f(m,n)\equiv 0\pmod m}N^\epsilon,\\
\end{eqnarray*}
where $\tau$ is the divisor function and we have used the fact that $f$ is irreducible so $f(m,n)\ne 0$.  

Let $f_0$ be the coefficient of $n^{\deg f}$ in $f$.  We have 
$$\twosum{m\leq N}{m|f_0}\chi(m)\twosum{n\leq N}{f(m,n)\equiv 0\pmod m}N^\epsilon\ll_f N^{1+\epsilon}.$$
If a prime $m$ does not divide $f_0$ but  $m|f(m,n)$ then we must have
$m|n$.  Therefore 
$$\twosum{m\leq N}{m\nmid f_0}\chi(m)\twosum{n\leq N}{f(m,n)\equiv
  0\pmod m}N^\epsilon=\twosum{m\leq N}{m\nmid f_0}\chi(m)\twosum{n\leq N}{m|n}N^\epsilon\ll_\epsilon N^{1+\epsilon}.$$
The result follows.
\end{proof}

Our proof of Theorem \ref{formlod} begins by applying methods from the geometry of numbers, similar to those employed by Daniel in \cite{daniel}.  We call a point $(m,n)$ primitive modulo $d$ if $(m;n;d)=1$.  We say that the primitive points $(m_1,n_1)$ and $(m_2,n_2)$ are equivalent modulo $d$ if 
$$(m_2,n_2)\equiv \lambda(m_1,n_1)\pmod d$$
for some $\lambda\in\Z$ which must necessarily satisfy $(\lambda;d)=1$.  We observe that the property $f(m,n)\equiv 0\pmod d$ is preserved by equivalence so we may let $\mathcal U(d)$ be the set of equivalence classes mod $d$ for which it holds.

For each $x\in\mathcal U(d)$ we let $\lambda(x)$ be the lattice in $\Z^2$ generated by the points of $x$.  Thus if we fix an $(m,n)\in x$ then $\lambda(x)$ consists of all the points congruent mod $d$ to some multiple of $(m,n)$.  It follows that $\det \lambda(x)=d$ and that the set of primitive points in $\lambda(x)$ is precisely $x$. Each primitive solution of $f(m,n)\equiv 0\pmod d$ occurs in precisely one lattice $\lambda(x)$ but a nonprimitive solution may occur in more than one.  Since any nonprimitive point has $(m;d)>1$ and $\#\mathcal U(d)\ll_{\epsilon,f} d^\epsilon$, (see for example Daniel \cite[(3.5)]{daniel}), we can handle this multiplicity issue with the last lemma.

We let $\mathcal U'(d)$ be the subset of $\mathcal U(d)$ containing those $x$ generated by a point $(m,n)$ with $(m;d)=1$.  If $x\notin\mathcal U'(d)$ then all $(m,n)\in \lambda(x)$ have $(m;d)>1$.  It is clear that $\#\mathcal U'(d)=\nu(d)$.  We can therefore deduce using the last lemma that 
$$\sum_{d\sim D}|A_d(N,\alpha)-M_d(N,\alpha)|$$
$$\ll_\epsilon N^{1+\epsilon}+\sum_{d\sim D}\sum_{x\in\mathcal U'(d)}\left|\sum_{(m,n)\in \lambda(x)\cap (0,N]\times \Z}\alpha_mW(\frac{n}{N})-\frac{N\hat W(0)}{d}\sum_{m\leq N}\alpha_m\right|.$$
We must therefore bound 
$$S=\sum_{d\sim D}\sum_{x\in\mathcal U'(d)}\left|\psi(\lambda(x),N,\alpha)-\frac{N\hat W(0)}{d}\sum_{m\leq N}\alpha_m\right|$$
where $\psi$ is the quantity studied in the last section.

We let $B_1(x),B_2(x)$ denote the minimal basis of $\lambda(x)$ and write $B(x)$ for the matrix with rows the $B_i$.  If $D\geq N^{\delta_1}$ it is necessary to remove from $S$ any lattices  for which $B_{11}$ is unusually small, say $B_{11}(x)\leq D^{1/2-\eta}$ for some $\eta>0$.  For these lattices we  bound the sums
$$\sum_{d\sim D}\twosum{x\in\mathcal U'(d)}{B_{11}(x)\leq D^{1/2-\eta}}\left|\psi(\lambda(x),N,\alpha)\right|$$
and 
$$\sum_{d\sim D}\twosum{x\in\mathcal U'(d)}{B_{11}(x)\leq D^{1/2-\eta}}\left|\frac{N\hat W(0)}{d}\sum_{m\leq N}\alpha_m\right|.$$
The first sum is bounded by 
$$S_1=\sum_{d\sim D}\twosum{x\in\mathcal U'(d)}{B_{11}(x)\leq D^{1/2-\eta}}\#(\lambda(x)\cap [0,N]^2)$$
whilst the second is at most of order
$$S_2=\frac{N^2}{D}\sum_{d\sim D}\twosum{x\in\mathcal U'(d)}{B_{11}(x)\leq D^{1/2-\eta}}1.$$
We estimate these using the following lemma.  

\begin{lem}\label{smallcount}
Suppose $0\ne (u,v)\in\Z^2$.  Then for any $\epsilon>0$ we have 
$$\#\{(d,x):d\sim D,x\in\mathcal U'(d),(u,v)\in \lambda(x)\}\ll_\epsilon \|(u,v)\|^\epsilon.$$
\end{lem}
\begin{proof}
Since $f$ is irreducible and $(u,v)\ne 0$ we know that $f(u,v)\ne 0$.  The number of possible $d$ is then bounded by 
$$\tau(f(u,v))\ll_\epsilon \|(u,v)\|^\epsilon.$$
For each such $d$ the number of possible $x$ cannot exceed $\nu(d)=O_\epsilon(d^\epsilon)$.  The result follows.
\end{proof}

Recall that $\det\lambda(x)\sim D$.  Therefore, if $B_{11}(x)\leq D^{1/2-\eta}$ we must have $B_{11}(x)=(u,v)$ for some $0\ne(u,v)\in\Z^2$ with $u\leq D^{1/2-\eta}$ and $v\ll D^{1/2}$.  It follows that the number of terms in our sums $S_1,S_2$ is at most $O_\epsilon(D^{1-\eta}N^\epsilon)$. We immediately   deduce that 
$$S_2\ll_\epsilon N^{2+\epsilon}D^{-\eta}.$$
To bound $S_1$ we use the standard estimate for the number of lattice points to get 
$$S_1\ll\sum_{d\sim D}\twosum{x\in\mathcal U'(d)}{B_{11}(x)\leq D^{1/2-\eta}}\left(\frac{N^2}{d}+\frac{N}{\|B_1(x)\|}+1\right).$$
From the above discussion we obtain the bounds 
$$\sum_{d\sim D}\twosum{x\in\mathcal U'(d)}{B_{11}(x)\leq D^{1/2-\eta}}\frac{N^2}{d}\ll_\epsilon N^{2+\epsilon}D^{-\eta}$$
and 
$$\sum_{d\sim D}\twosum{x\in\mathcal U'(d)}{B_{11}(x)\leq D^{1/2-\eta}}1\ll_\epsilon D^{1-\eta}N^\epsilon.$$
Finally we use Lemma \ref{smallcount} to get 
\begin{eqnarray*}
\lefteqn{\sum_{d\sim D}\twosum{x\in\mathcal U'(d)}{B_{11}(x)\leq D^{1/2-\eta}}\frac{N}{\|B_1(x)\|}}\\
&\ll_\epsilon&N^{1+\epsilon}\sum_{0<\|(u,v)\|\ll D^{1/2}}\frac{1}{\sqrt{u^2+v^2}}\\
&\ll_\epsilon&N^{1+\epsilon}D^{1/2}.\\
\end{eqnarray*}
We conclude that 
$$S_1+S_2\ll_\epsilon N^{1+\epsilon}D^{1/2}+D^{1-\eta}N^\epsilon+N^{2+\epsilon}D^{-\eta}$$
so this bound also holds for the contribution to $S$ from lattices with $B_{11}\leq D^{1/2-\eta}$.  Recalling that $N^{\delta_1}\leq D\leq N^{4/3-\delta_1}$ we see that if we take a  small enough $\epsilon$ then this bound is $O(N^{2-\delta_2})$ for $\delta_2>0$ sufficiently small in terms of $\delta_1$ and $\eta$.  It should be noted that the exponent $\frac{4}{3}$ is not critical for this part of the argument.

It remains to consider 
$$S_3=\sum_{d\sim D}\twosum{x\in\mathcal U'(d)}{B_{11}(x)>D^{1/2-\eta}}\left|\psi(\lambda(x),N,\alpha)-\frac{N\hat W(0)}{d}\sum_{m\leq N}\alpha_m\right|,$$
to which we will apply Theorem \ref{latticelargesieve}.  If $D\geq
N^{\delta_1}$ then $\eta$ is a quantity that we can take arbitrarily
small, whereas if $D\leq N^{\delta_1}$ then we shall take $\eta>1/2$,
(so that all lattices are included).

If $x\in \mathcal U'(d)$ then $\lambda(x)$ consists of all points congruent modulo $d$ to a multiple of some $(m,n)$ with $(m;d)=1$.  It follows that the $m$-coordinates of points in $\lambda(x)$ are coprime.  The sum is over lattices $\lambda(x)$ which have $\det\lambda(x)\sim D$ and $D^{1/2-\eta}<B_{11}(x)\ll D^{1/2}$.  For each possible value of $B_{11}(x)$ in this range there are $O(D^{1/2})$ permissible values for $B_{12}(x)$.  It follows by Lemma \ref{smallcount} that the number of lattices in the sum with any  given value of $B_{11}$ is at most $O_\epsilon(D^{1/2}N^\epsilon)$.  We therefore subdivide $S_3$ into $O_\epsilon(N^\epsilon)$ dyadic intervals depending on the size of $B_{11}$ and then subdivide each dyadic sum into $O_\epsilon(D^{1/2}N^\epsilon)$ subsums in which each possible value of $B_{11}$ occurs at most once.  The resulting subsums may be estimated using Theorem \ref{latticelargesieve}.  Suppose $\delta>0$.  If $D\leq N^{1-\delta}$ we get 
$$S_3\ll_{\epsilon,A} D^{1/2}N^{\epsilon-A},$$
for any $A\in\N$, which is certainly small enough.  If $D\geq N^{1-\delta}$ we must check the condition 
$$D\leq M_1N^{1-\delta}.$$
However $M_1\geq D^{1/2-\eta}$  so it is sufficient that 
$$D\leq N^{\frac{1-\delta}{1/2+\eta}}.$$
Since $D\leq N^{4/3-\delta_1}$ this is certainly satisfied if we take
$\delta,\eta$ small enough. (Since $D\geq N^{1-\delta}$ we are in the
case in which any $\eta>0$ is admissible).  
We may therefore deduce from Theorem \ref{latticelargesieve} that 
$$S_3\ll_\epsilon N^\epsilon D^{1/2}\cdot N^{1+2\delta+\epsilon}D^{1/4+\eta/2}\ll_\epsilon N^{1+2\delta+\epsilon}D^{3/4+\eta/2}.$$
Since $D\leq N^{4/3-\delta_1}$ we see that if we take $\delta,\epsilon$ and $\eta$ sufficiently small in terms of $\delta_1$ then 
$$S_3\ll_{\delta_1}N^{2-\delta_2}$$
for some $\delta_2>0$.  This is where the value $4/3$ is critical as for larger $D$ we do not get a nontrivial bound from Theorem \ref{latticelargesieve}.

We conclude that 
$$\sum_{d\sim D}|A_d(N,\alpha)-M_d(N,\alpha)|\ll_{\delta_1} N^{2-\delta_2}$$
for some $\delta_2>0$, thus completing the proof of Theorem \ref{formlod}.

When we apply the weighted sieve in the next section we will use the following upper bound to show that not too many values of $f$ are divisible by the square of a prime.

\begin{lem}\label{primesquare}
Let $\alpha_m$ be a sequence of complex numbers with $|\alpha_m|\leq 1$.  For any $\delta_1>0$ there exists a $\delta_2>0$, depending only on $\delta_1$, such that 
$$\sum_{N^{\delta_1}\leq p\leq N^{2-\delta_1}}|A_{p^2}(N,\alpha)|\ll_{\delta_1} N^{2-\delta_2},$$
the sum being over primes $p$.
\end{lem}

\begin{proof}
We have 
$$A_{p^2}(N,\alpha)=\twosum{(m,n)\in (0,N]\times\Z}{f(m,n)\equiv 0\pmod {p^2}}\alpha_m W(\frac{n}{N})\ll \#\{(m,n) \in [0,N]^2:f(m,n)\equiv 0\pmod{p^2}\}.$$
If $f(m,n)\equiv 0\pmod{p^2}$ then $(m,n)\in \lambda(x)$ for at least one $x\in\mathcal U(p^2)$.  It follows that 
$$A_{p^2}(N,\alpha)\ll\sum_{x\in\mathcal U(p^2)}\#(\lambda(x)\cap [0,N]^2).$$ 
We may bound this by 
$$\sum_{x\in \mathcal U(p^2)}(\frac{N^2}{p^2}+\frac{N}{\|B_1(x)\|}+1).$$
Using that $\#\mathcal U(p^2)\ll_\epsilon N^\epsilon$ we have 
$$\sum_{N^{\delta_1}\leq p\leq N^{2-\delta_1}}\sum_{x\in\mathcal U(p^2)}(\frac{N^2}{p^2}+1)\ll_\epsilon N^{2-\delta_1+\epsilon}.$$
It therefore remains to estimate 
$$N\sum_{N^{\delta_1}\leq p\leq N^{2-\delta_1}}\sum_{x\in \mathcal U(p^2)}\frac{1}{\|B_1(x)\|}.$$
If points are equivalent modulo $p^2$ then they must also be
equivalent modulo $p$.  It follows that if $x\in \mathcal U(p^2)$ then
there is some $x'\in \mathcal U(p)$ with $\lambda(x)\subseteq
\lambda(x')$.  Different equivalence classes in $\mathcal U(p^2)$ may
give rise to the same class in $\mathcal U(p)$ but the total number of
times a class may occur cannot exceed $\#\mathcal U(p^2)\ll_\epsilon
N^\epsilon$.  Our sum is therefore majorised by
$$N^{1+\epsilon}\sum_{N^{\delta_1}\leq p\leq N^{2-\delta_1}}\sum_{x\in \mathcal U(p)}\frac{1}{\|B_1(x)\|}.$$
To estimate this final sum we use part of Daniel's proof of
\cite[Lemma 3.2]{daniel}; which is very similar to our above
derivation of a bound on $S_1$.  Specifically, if we set $Q=N^{2-\delta_1}$, our sum is bounded by the quantity $T_1^*(Q)$ defined in that proof so it is $O_\epsilon(N^{1-\delta_1/2+\epsilon})$.  We therefore conclude that
$$N^{1+\epsilon}\sum_{N^{\delta_1}\leq p\leq N^{2-\delta_1}}\sum_{x\in \mathcal U(p)}\frac{1}{\|B_1(x)\|}\ll_\epsilon N^{2-\delta_1/2+\epsilon}.$$
The result follows on combining the above estimates and taking $\delta_2<\delta_1/2$.  
\end{proof}

\section{Proof of Theorem \ref{almostprimeform}}

We will sieve the sequence $\ca=(a_l)$ given by 
$$a_l=\twosum{(m,n)\in (0,N]\times\Z}{|f(m,n)|=l}\chi(m)W(\frac{n}{N}).$$
This is supported on $l\ll_f N^k$ and by Theorem \ref{formlod} we know
that it has level of distribution $N^\theta$ for any $\theta<\frac43$.
Since $f$ is irreducible we deduce from the prime ideal theorem that
the values $\nu(p)$ are $1$ on average and we may therefore use a
$1$-dimensional weighted sieve.  By assumption we know that $\nu(p)<p$
for all primes $p$. It can therefore be shown that
$$\prod_{p<z}(1-\frac{\nu(p)}{p})=\frac{c_f+o(1)}{\log z}$$
for some $c_f>0$.  

We use the weighted sieve as described by Greaves in \cite[Chapter 5]{greavesbook}.  If $r\geq 2$ we deduce that if 
$$\frac{3}{4}k<r-\delta_r$$
then for all sufficiently large $N$ we have 
$$\sum_l{}^*\, a_l\gg \frac{N^2}{(\log N)^2},$$
where $\sum^*$ denotes a sum over certain $l$ which are the product of at most $r$ distinct primes.  Specifically, \cite[Section 5.2]{greavesbook} shows that we can take $\delta_r=0.144001\ldots$.  The above estimate therefore follows if 
$$r>\frac{3}{4}k+0.15$$
which is equivalent to $r\geq [3k/4]+1$.  Observe that it is essential that we had $\delta_r<\frac14$. The
simplest form of the weighted sieve \cite[Section 5.1]{greavesbook}
would therefore have been insufficient.

It remains to show that we can produce numbers with at most $r$ prime factors when counted with multiplicity.  Examining the construction of the sieve it can be seen that there are constants $0<\alpha<\beta<2$, depending on $r$, such that $\sum^*$ is actually a sum over $l$ all of whose prime factors exceed $N^\alpha$ and for which 
$$\twosum{p|l}{p\leq N^\beta}1+\sum_{p\geq N^\beta}\sum_{a:\,p^a|l}1\leq r.$$
This means that only prime factors smaller than $N^\beta$ are counted without multiplicities.  We can deduce from Lemma \ref{primesquare} that the contribution of $l$ which are divisible by $p^2$ for $p\in [N^\alpha,N^\beta]$ is $O(N^{2-\delta})$ for some $\delta>0$ depending on $\alpha$ and $\beta$.  We may therefore conclude that for all sufficiently large $N$ we have 
$$\sum_{l\in P_r}a_l\gg \frac{N^2}{(\log N)^2}$$
thereby completing the proof of Theorem \ref{almostprimeform}.

\newpage

\addcontentsline{toc}{section}{References} 
\bibliographystyle{plain}
\bibliography{../biblio}

\bigskip
\bigskip

Mathematical Institute,

University of Oxford,

Andrew Wiles Building, 

Radcliffe Observatory Quarter, 

Woodstock Road, 

Oxford 

OX2 6GG 

UK
\bigskip

{\tt irving@maths.ox.ac.uk}

\end{document}